\documentclass[11pt,a4paper,reqno]{amsart}
\usepackage[english]{babel}
\usepackage[applemac]{inputenc}
\usepackage[T1]{fontenc}
\usepackage{palatino}
\usepackage{amsmath}
\usepackage{amssymb}
\usepackage{amsthm}
\usepackage{amsfonts}
\usepackage{graphicx}
\usepackage{color}
\usepackage{vmargin}
\usepackage{mathtools}

\setmarginsrb{20mm}{20mm}{20mm}{20mm}{10mm}{10mm}{10mm}{10mm}

\theoremstyle{plain}
\newtheorem{theorem}{Theorem}[section]
\newtheorem{prop}[theorem]{Proposition}
\newtheorem{lemma}[theorem]{Lemma}
\newtheorem{corol}[theorem]{Corollary}
\theoremstyle{definition}
\newtheorem{definition}{Definition}[section]

\newtheorem{remark}{\textnormal{\textbf{Remark}}}[section]
\newtheorem*{remark*}{\textnormal{\textbf{Remark}}}

\def\vint{\mathop{\mathchoice%
          {\setbox0\hbox{$\displaystyle\intop$}\kern 0.22\wd0%
           \vcenter{\hrule width 0.6\wd0}\kern -0.82\wd0}%
          {\setbox0\hbox{$\textstyle\intop$}\kern 0.2\wd0%
           \vcenter{\hrule width 0.6\wd0}\kern -0.8\wd0}%
          {\setbox0\hbox{$\scriptstyle\intop$}\kern 0.2\wd0%
           \vcenter{\hrule width 0.6\wd0}\kern -0.8\wd0}%
          {\setbox0\hbox{$\scriptscriptstyle\intop$}\kern 0.2\wd0%
           \vcenter{\hrule width 0.6\wd0}\kern -0.8\wd0}}%
          \mathopen{}\int}
\def\vintslides{\mathop{\mathchoice%
          {\setbox0\hbox{$\displaystyle\intop$}\kern 0.22\wd0%
           \vcenter{\hrule height 0.04em width 0.6\wd0}\kern -0.82\wd0}%
          {\setbox0\hbox{$\textstyle\intop$}\kern 0.2\wd0%
           \vcenter{\hrule height 0.04em width 0.6\wd0}\kern -0.8\wd0}%
          {\setbox0\hbox{$\scriptstyle\intop$}\kern 0.2\wd0%
           \vcenter{\hrule height 0.04em width 0.6\wd0}\kern -0.8\wd0}%
          {\setbox0\hbox{$\scriptscriptstyle\intop$}\kern 0.2\wd0%
           \vcenter{\hrule height 0.04em width 0.6\wd0}\kern -0.8\wd0}}%
          \mathopen{}\int}

\def \div {{\rm div}}

\newcommand*{\norm}[1]{\left\Vert{#1}\right\Vert}
\newcommand*{\abs}[1]{\left\vert{#1}\right\vert}

\newcommand*{\Lip}{\mathrm{Lip}\,}

\newcommand*{\dist}{\mathrm{dist}}

\newcommand{\spt}{\operatorname{spt}}

\newcommand{\N}{{\mathbb N}}
\newcommand{\R}{{\mathbb R}}

\newcommand{\ep}{\epsilon}
\newcommand{\Om}{\Omega}

\newcommand{\Hei}{{\mathbb{H}}_{1}}
\newcommand{\Hein}{{\mathbb H}_{n}}

\newcommand{\LIP}{\operatorname{LIP}}

\newcommand{\Ha}{\mathcal{H}}
\newcommand{\ud}{\mathrm {d}}

\definecolor{blau}{rgb}{0.1,0.0,0.9}

\newcommand{\red}{\color{red}}
\newcounter{komcounter}
\numberwithin{komcounter}{section}

\newcommand{\kom}[1]{}
\renewcommand{\kom}[1]{$\backslash$ {\bf  \red #1}$\backslash$}

\newcommand{\komTT}[1]{{\bf \red [\refstepcounter{komcounter}\thekomcounter (by Tomasz): #1]}}

\title[AMV harmonic functions in subriemannian and RCD settings]{Asymptotically mean value harmonic functions in subriemannian and RCD settings}

\author[Tomasz Adamowicz]{Tomasz Adamowicz{\small$^1$}}
\address{T.A.: Institute of Mathematics, Polish Academy of Sciences,
\'Sniadeckich 8, Warsaw, 00-656, Poland\/}
\email{tadamowi@impan.pl}

\author[Antoni Kijowski]{Antoni Kijowski}
\address{A.K.: Analysis on Metric Spaces Unit, Okinawa Institute of Science and Technology Graduate University,
Okinawa, 904-0495, Japan}
\email{antoni.kijowski@oist.jp}

\author[Elefterios Soultanis]{Elefterios Soultanis{\small$^2$}}
\address{E.S.: Radboud University, IMAPP, Heyendaalseweg 135, 6525 AJ Nijmegen, Netherlands\/}
\email{elefterios.soultanis@gmail.com}

\begin{document}

\begin{abstract}
We consider weakly and strongly asymptotically mean value harmonic (amv-harmonic) functions on subriemannian and RCD settings. We demonstrate that, in non-collapsed RCD-spaces with vanishing metric measure boundary, Cheeger harmonic functions are weakly amv-harmonic and that, in Carnot groups, weak amv-harmonicity equivalently characterizes harmonicity in the sense of the sub-Laplacian.

In homogeneous Carnot groups of step $2$, we prove a Blaschke--Privaloff--Zaremba type theorem. Similar results are discussed in the settings of Riemannian manifolds and for Alexandrov surfaces.
%
\end{abstract}

\maketitle

\footnotetext[1]{T. Adamowicz was supported by a grant of National Science Center, Poland (NCN),
 UMO-2017/25/B/ST1/01955.}
\footnotetext[2]{E. Soultanis was supported by Swiss National Foundation grant  no.  182423.}

\setcounter{tocdepth}{1}
	
\section{Introduction}

%


A key problem in the theory of harmonic functions is to find various characterizations of harmonicity. A particularly strong example of this is the mean value property in the Euclidean setting. Unfortunately, harmonic functions rarely enjoy the mean value property outside Euclidean space. For example, in the setting of Riemannian manifolds it is known to hold on manifolds which are either flat or rank-one symmetric  and in dimensions 2-5, see the related discussion on the Lichnerowicz conjecture in Example 4 in~\cite{agg} and references therein. However a suitable variant beyond the Euclidean setting turns out to be the \emph{asymptotic mean value (amv) property}
\begin{align}\label{eq:rlap}
	\Delta_ru\stackrel{r\to 0}{\longrightarrow }0,\quad \Delta_ru(x):=\frac{u_{B_r(x)}-u(x)}{r^2}.
\end{align}

In Riemannian manifolds the convergence in \eqref{eq:rlap} is locally uniform for a harmonic function $u$. However the \emph{$r$-laplacian} in \eqref{eq:rlap} can be defined for locally integrable functions on any \emph{metric measure space} $X=(X,d,\mu)$, where $u_{B_r(x)}$ stands for the mean-value of $u$ over a ball $B_r(x)$ with radius $r>0$ centered at $x\in X$. Consequently the amv property can be formulated in a very general setting (note that the mean value (mv) property can be also expressed in the same generality in terms of vanishing $r$-laplacian, see see Definition \ref{main-def1} and also \cite{agg, gg}). Let us briefly mention that the $r$-laplacian arises in approximation problems of Riemannian manifolds by graphs~\cite{bik}, in geometric group theory~\cite{kleiner}, and in relation with functions of bounded variation~\cite[Theorem 1]{co}. It moreover appears in connection with $p$-harmonic type equations and stochastic games~\cite{mpr, mpr2, arlo}, as well as in the setting of Carnot groups; the $r$-laplacian and its relations to the subelliptic harmonic functions have been studied, for instance, in~\cite{adw, flm, fp}.

For recent work on amv functions on general metric measure spaces we refer the reader to~\cite{mi-te, mi-te2} and in particular to~\cite{aks} where we focus on functions with the amv property in doubling metric measure spaces and show their local H\"older regularity and the mv property for their blow-ups. As an important special case, in~\cite{aks} we consider weighted Euclidean setting and find an elliptic PDE satisfied by amv-harmonic functions.


\subsection{Main results}
In this paper we continue the investigations of and complement results from~\cite{aks} by considering functions with the amv property in subriemannian and RCD settings. As in \cite{aks} we consider two notions of \emph{asymptotically mean value harmonic} (amv-harmonic) functions, arising from different ways to interpret the limit in \eqref{eq:rlap}. More precisely, we define \emph{strong} and \emph{weak} amv-harmonicity as follows. 
\begin{definition}[see Definition 1.1 in~\cite{aks}]\label{main-def1}
	Let $X=(X,d,\mu)$ be a metric measure space.
	\begin{itemize}
		\item[(1)] A function $u\in L^1_{loc}(X)$ is \emph{strongly amv-harmonic}, if
		\[
		\lim_{r\to 0}\|\Delta_r u\|_{L^\infty(K)}=0
		\]
		for any compact set $K\subset X$;
		\item[(2)]A function $u\in L^2_{loc}(X)$ on a metric measure space $X$ is \emph{weakly amv-harmonic} if
		\begin{equation*}
		\lim_{r\to 0}\int_X\varphi\Delta_ru\ud\mu=0
		\end{equation*}
		for every compactly supported Lipschitz function $\varphi$.
	\end{itemize}
\end{definition}

To state our first result we refer to Section \ref{sec:rcd} and the references therein for the theory of RCD-spaces, which include GH-limits of Riemannian manifolds with uniform Ricci lower bounds and dimension upper bounds (i.e. \emph{Ricci limit spaces}) as well as Alexandrov spaces equipped with the Hausdorff measure of appropriate dimension. 

We refer to Definition \ref{def:boundary} for \emph{the boundary} $\partial X$ of an RCD$(K,N)$-space $X=(X,d,\mu)$ and recall $X$ is \emph{non-collapsed} if $\mu=\Ha^N$, the $N$-Hausdorff measure. 
Below and in what follows, \emph{harmonic} functions mean continuous local minimizers of the Cheeger energy, cf. Definition \ref{def:ch-harm}.

\begin{theorem}\label{thm:ncRCD}
	Let $\Omega\subset X$ be a domain in a non-collapsed $RCD(K,N)$-space $X=(X,d,\Ha^N)$ and suppose $\Omega\cap \partial X=\varnothing$. Then (Cheeger) harmonic functions in $\Omega$ are weakly amv-harmonic in $\Omega$.
\end{theorem}

 In particular, Theorem~\ref{thm:ncRCD} implies that harmonic functions on Alexandrov spaces without boundary are weakly amv-harmonic. 
For locally Lipschitz continuous functions, the converse implication in Theorem \ref{thm:ncRCD} also holds, see Remark \ref{rmk:rcd}. 

It is not known whether harmonic functions are also \emph{strongly} amv-harmonic in the setting of Theorem \ref{thm:ncRCD} (see \cite[Corollary 3.5]{mon-naber19} for a related weaker statement). Such a result would need very good control over the rate of change of the $\Ha^N$-measure of balls centered at nearby points: the presence of a weight (e.g. the setting of \emph{weakly} non-collapsed RCD-spaces, see \cite{gig-dephil18}) rules out the possibility of having even the weak amv property for harmonic functions, see \cite[Section 6]{aks}.

In Carnot groups, homogeneity properties come into play and ensure the good control needed for the strong amv property. Using a second order Taylor expansion our next theorem establishes the equivalence of harmonic and strongly amv-harmonic functions in Carnot groups equipped with homogeneous pseudonorms with sufficiently regular unit balls. 

To state the theorem recall that an $s$-step Carnot group $\mathbb G$ with Lie algebra $\mathfrak{g}=V_1\oplus\ldots \oplus V_s$ can be identified with $\R^m$, where $m=v_1+\cdots+v_s$ and $v_i:=\dim V_i$, for $i=1,\ldots, s$ (see Section~\ref{sec:Carnot} for more details on Carnot groups). We denote by  $z=(z^{(1)},z^{(2)},\ldots,z^{(s)})$ the coordinates in $\mathbb{G}$, where $z^{(i)}\in \R^{v_i}$. We consider a pseudonorm $\rho$ on $\mathbb G$ such that 
\begin{equation}\label{eq:pseudof}
	\rho(z):=F(\|z^{(1)}\|,z'),\quad z=(z^{(1)},z')\in \mathbb G,
\end{equation}
for some $F:\R\times \R^{m-v_1}\to \R$ that is $C^1$ with $\partial_1 F>0$ outside the origin.


\begin{theorem}\label{cor:carnot}
	Let $\Omega\subset \mathbb G$ be an open domain in a Carnot group $\mathbb G$ of homogeneous dimension $Q$ equipped with a  pseudonorm $\rho$ satisfying \eqref{eq:pseudof}. If $\Omega$ denotes the metric measure space $(\Omega,\rho,\Ha^Q)$, where $\Ha^Q$ is the Haar measure on $\mathbb G$, then the following are equivalent for a function $u\in HW^{1,2}_{loc}(\Omega)$.
	\begin{itemize}
		\item[(i)] $u$ is strongly amv-harmonic;
		\item[(ii)] $u$ is weakly amv-harmonic;
		\item[(iii)] $u$ satisfies $\Delta_\mathbb Gu=0$ in $\Omega$.
	\end{itemize}
\end{theorem}
In the claim above $HW^{1,2}_{loc}(\Omega)$ denotes the horizontal Sobolev space and $\Delta_\mathbb G$ is the subelliptic laplacian \eqref{eq:subelliptic_laplacian}. Theorem \ref{cor:carnot} characterizes amv-harmonicity (in both senses) for a large class of pseudonorms, namely those satisfying \eqref{eq:pseudof}, and in particular for the Koranyi gauge \eqref{G-norm}. See \cite{fp} for similar results for $p$-harmonic functions on Carnot groups with the Koranyi gauge. We remark that solutions to the sub-Laplacian equation are known to satisfy certain \emph{weighted} mean value properties, see~\cite[Chapter 5.5]{blu} and~\cite{adw}.

We close this subsection by restricting ourselves to Carnot groups of step $s=2$ where we are able to obtain a Blaschke--Privaloff--Zaremba (BPZ) type theorem. Such results update regularity from continuous to smooth under the assumption that the $r$-laplacian vanishes \emph{pointwise} as $r\to 0$. In the theorem below we consider the Koranyi gauge \eqref{G-norm} on a Carnot group.

\begin{theorem}\label{BPZ-thm-Intro}
	Let $\Om\subset \mathbb G$ be a domain in the Carnot group $\mathbb G$ of step $2$ equipped with the gauge distance $d_{\mathbb G}$ and let $f:\Om\to \R$ be a continuous function satisfying
	\[
	\lim_{r\to 0^{+}}\Delta_rf(x)=0,\quad x\in\Om.
	\]
	Then $f$ is a sub-elliptic harmonic function, i.e. for all $x\in \Om$ it holds that 
	\[
	\lim_{r\to 0^{+}}\Delta_r f(x)=C\Delta_{\mathbb G} f(x)=0,
	\]
	where $C$ depends on $v_1$ (the dimension of the horizontal layer $V_1$) and the geometry of $\mathbb G$.
\end{theorem}
Similar results are well-known on Riemannian manifolds (cf. Proposition~\ref{BPZ-thm2} for a sketch of proof for the readers convenience). We also obtain BPZ-type results for Alexandrov surfaces (cf. Proposition~\ref{prop:Alex-surf}), and for  Alexandrov spaces (cf.  Proposition~\ref{BPZ-thm3}). In summary, we show that continuity and the strong amv-harmonicity are sufficient conditions to imply the harmonicity with respect to various natural Laplace operators in some singular spaces.



\subsection{Overview and strategy}
One of the drawbacks of the $r$-laplacian $\Delta_r$ is its lack of self-ad\-joint\-ness: without further assumptions on the geometry of the underlying space, the difference $\Delta_r-\Delta^*_r$ of the $r$-laplacian and its formal adjoint (cf. Section \ref{sec-dbl}) does not tend to zero in any reasonable sense. In the RCD-setting, where the translation invariance of the measure is not available, a large part of our work is to deal with this issue. We overcome it by introducing the \emph{symmetrized $r$-laplacian} of a locally integrable function $u$ defined in a metric measure space  $X=(X,d,\mu)$  (see Definition~\ref{eq:symkernel} for details):
\begin{align}\label{intro-sym-lapl}
\widetilde\Delta_ru(x):=\frac 12\int_{B_r(x)}\left[\frac{1}{\mu(B_r(x))}+\frac{1}{\mu(B_r(y))}\right]\frac{u(y)-u(x)}{r^2}\ud\mu(y), \quad x\in X.
\end{align}
In Section~\ref{sec:modified} we study the basic properties of the symmetrized $r$-laplacian, in particular self-ad\-joint\-ness, which makes it a viable candidate for a yet another laplacian studied beyond the Euclidean setting. Crucially, the symmetrized $r$-laplacian is connected to the (approximate) Korevaar--Schoen (KS) energy so that, if the KS-energy is a Dirichlet form in the sense of Definition \ref{def:ksexist}, then the symmetrized $r$-laplacians $\widetilde\Delta_ru$ converge in a suitable weak sense, as $r\to 0$, to the Korevaar--Schoen laplacian $\Delta_{KS}u$ of $u$ for any $u\in D(\Delta_{KS})$, see Section \ref{sec:ks} and Proposition \ref{prop:modified}. It follows that local minimizers of the Korevaar--Schoen energy are characterized as functions whose symmetrized $r$-laplacian vanishes in the weak sense as $r\to 0$. 

Both RCD-spaces and Carnot groups have the property that the Korevaar--Schoen energy is a Dirichlet form. For RCD-spaces, this follows from \cite{gt20} (in the non-collapsed case also from work of Sturm \cite{st1,sturm2}), while for Carnot groups it is implied by \cite{tan}. See Propositions \ref{prop:gigtyu} and \ref{prop:carnot} for precise statements. For Carnot groups (where $\widetilde \Delta_r=\Delta_r$) this observation leads to the equivalence of weak amv-harmonicity and subelliptic harmonicity. 
We remark that the novel implication in Theorem \ref{cor:carnot} is that weakly harmonic functions are subelliptic harmonic; indeed, the strong amv property of subelliptic harmonic functions follows with a straightforward calculation as in \cite[Lemma 4.1]{fp}.

In the RCD-setting the equality $\widetilde \Delta_r=\Delta_r$ breaks down and, in the presence of boundary, the two $r$-laplacians do not agree even asymptotically (which can be seen by considering e.g. the standard half plane). To circumvent this problem we use the notion of \emph{vanishing metric measure boundary} introduced in \cite{klp17}, which has recently been shown in \cite{BMS22} to be equivalent to having no boundary in the sense of \cite{gig-dephil18}. In non-collapsed RCD-spaces the vanishing mm-boundary enables us to obtain the asymptotic vanishing of $\widetilde \Delta_r-\Delta_r$ as $r\to 0$ in a suitable weak sense, cf. Lemma \ref{lem:nodev} and Corollary \ref{cor:amv-sa}. This yields the proof of Theorem \ref{thm:ncRCD}.

Theorem \ref{BPZ-thm-Intro}, and the BPZ-type theorems in Section \ref{sec:bpz}, are obtained by modifying the original argument in \cite{pri} to fit the various nonsmooth settings. The crucial tool is the solvability of the Dirichlet problem with continuous data on the boundary of the metric ball for various notions of Laplacians on singular spaces, and lower estimates for the Laplacian of distance functions. An inspection of the proof of Theorem \ref{BPZ-thm-Intro} shows that the restriction to step 2 arises from regularity issues of the boundary of the unit ball $B_\mathbb G$ with respect to the Koranyi gauge: in step $s>2$ the solvability of the Dirichlet problem with continuous boundary data in $B_\mathbb G$ is an open problem.


\smallskip

{\bf Acknowledgements} The authors would like to thank Ben Warhurst for a discussion concerning Theorem~\ref{cor:carnot}.

\section{Preliminaries}

\subsection{Notation and conventions}\label{subs-notation}
Throughout this article, a metric measure space $X=(X,d,\mu)$ is a separable metric space $(X,d)$ equipped with a Radon measure $\mu$ that is finite and non-trivial on balls, i.e. $0<\mu(B)<\infty$ for all balls $B\subset X$.

Given a subset $F\subset X$ of a metric space and $r>0$, we denote
\[
N_r(F)=\{ x\in X:\ \dist(x,F)<r\} \textrm{ and }\overline N_r(F)=\{x\in X:\ \dist(x,F)\le r\}
\]
the open and closed $r$-neighbourhood of $F$ (note that $\overline N_r(F)$ need not be the closure of $N_r(F)$ unless $X$ is a length space). For $x\in X$, we denote by $B_r(x):=N_r(\{x\})$ and $\overline B_r(x):=\overline N_r(\{x\})$,  respectively, an open and closed ball centered at $x$ with radius $r$.

The Lipschitz constant of a map $f:(X, d_X) \to (Y, d_Y)$ between metric spaces is
\begin{equation*}
\LIP(f):=\sup_{x\ne y}\frac{d_Y(f(x),f(x))}{d_X(x,y)},
\end{equation*}
and the pointwise Lipschitz slope is
\begin{align*}
\Lip f(x):= \limsup_{r\to 0}\sup_{0<d(y,x)<r}\frac{d_Y(f(x),f(y))}{r},\quad x\in X.
\end{align*}
If $x\in X$ is an isolated point, then $\Lip f(x)=0$.

Given $u\in L_{loc}^1(X)$ and a $\mu$-measurable set $A\subset X$ with $\mu(A)>0$, 
\[
u_{A}=\vint_A u\ud\mu=\frac{1}{\mu(A)}\int_Au\ud\mu
\]
denotes the  average of $u$ over $A$.

\subsection{Doubling measures and averaging operators}\label{sec-dbl}

A measure $\mu$ on a separable metric space $X$ is called \emph{locally doubling} if, for every compact $K\subset X$, there exists $r_K>0$ and a constant $C_K>0$, such that 
\begin{equation}\label{eq:doubl}
\mu(B_{2r}(x))\le C_K\mu(B_r(x))
\end{equation}
for every $x\in K$ and $0<r\le r_K$.

Let $(X,d,\mu)$ be a metric measure space and $r>0$. Given a locally integrable function $u\in L^1_{loc}(X)$, we denote by
\[
A_r^\mu u(x)=A_r u(x)=\vint_{B_r(x)}u\ud\mu,\quad x\in X,
\]
the $r$-average function of $u$. Note that $A_ru(x)=u_{B_r(x)}$.
 
We will use the two notations interchangeably, depending on whether we want to view the average as a number, or an operator on a function space. Indeed, the function $A_ru:X\to\R$ is measurable, and $A_r$ defines a bounded linear operator $A_r:L^1(X)\to L^1(X)$ if and only if $a_r\in L^\infty(X)$, where
\[
a_r(x)=\int_{B_r(x)}\frac{\ud\mu(y)}{\mu(B_r(y))},\quad x\in X.
\]
Moreover, in this case the operator norm satisfies $\|A_r\|_{L^1\to L^1}=\|a_r\|_{L^\infty}$, see~\cite[Theorem 3.3]{Aldaz18}. This is true in particular when $\mu$ is a doubling measure.  On the other hand, it is true that by the Lebesgue differentiation theorem
\begin{align*}
u(x)=\lim_{r\to 0}A_ru(x)\,\textrm{ for almost every }x\in X,
\end{align*}
if $\mu$ is infinitesimally doubling, cf.~\cite[Remark 3.4.29]{HKST}.

If $X$ is doubling as a metric space, then there exists $C>0$ so that $\|A_r\|_{L^p\to L^p}\le C$ for every $r>0$ and every $1\le p<\infty$, cf.~\cite[Theorem 3.5]{Aldaz18}. However, $A_r$ is not a self-adjoint operator; the formal adjoint $A_r^*$ of $A_r$  is given by
\begin{equation*}
(A_r^\mu)^*u(x)=A_r^*u(x):=\int_{B_r(x)}\frac{u(y)\ud\mu(y)}{\mu(B_r(y))},\quad x\in X,
\end{equation*}
for $u\in L^1_{loc}(X)$. Indeed, a direct computation using the Fubini theorem yields that
\begin{equation}\label{eq:alaGreen}
\int_X vA_ru\ud\mu=\int_XuA_r^*v\ud\mu,\quad u\in L^p(X),\ v\in L^q(X),
\end{equation}
where $1/p+1/q=1$.

We may express the $r$-laplacian using the averaging operator as
\begin{align*}
\Delta_ru=\frac{A_ru-u}{r^2},\quad u\in L^1_{loc}(X).
\end{align*}
Next, we denote by
\begin{equation}\label{eq:adjointlaplace}
\Delta_r^*u:=\frac{A_r^*u-u}{r^2},\quad  u\in L^1_{loc}(X),
\end{equation}
the formal adjoint of the $r$-laplacian. Note that if $A_r:L^p(X)\to L^p(X)$ is bounded, then $\Delta_r:L^p(X)\to L^p(X)$ and $\Delta_r^*:L^q(X)\to L^q(X)$ are both bounded, where $1/p+1/q=1$.

\begin{remark}\label{rmk: Green}
 Identity \eqref{eq:alaGreen} together with~\eqref{eq:adjointlaplace} imply the following Green-type formula for  $u\in L^p(X),\ v\in L^q(X)$ where $1/p+1/q=1$:
 \[
  \int_X v \Delta_ru\ud\mu=\int_X u \Delta_r^*v\ud\mu.
 \]
The asymptotic behaviour of $\Delta_r-\Delta_r^*$ as $r\to 0$ plays an important role in Section \ref{sec:weak-amv}. 
\end{remark}

\begin{remark}\label{rmk:domain}
While most of our results will be formulated for metric measure spaces, they encompass the case of an open set $\Omega \subset X$ in the introduction. Indeed, an open subset $\Omega\subset X$ of a metric measure space can be regarded as a metric measure space $\Omega=(\Omega,d|_\Omega,\mu|_\Omega)$. In particular, \emph{if $X$ is locally doubling, then $\Omega$ is locally doubling.}
\end{remark}

\subsection{Cheeger and Korevaar--Schoen energy, and Dirichlet forms}\label{sec:ks}

In this section we recall some basic notions of calculus in metric measure spaces, needed in the discussion of the harmonic functions according to Cheeger, see Theorems~\ref{thm:ncRCD} and~\ref{cor:carnot}. Moreover, we discuss the Korevar--Schoen energy and its relation to the Dirichlet forms.
 
 One of the founding stones of the calculus theory in metric measure spaces are the notions of \emph{$p$-weak upper gradient} and the  \emph{minimal $p$-weak upper gradient}, see \cite[Chapter 5]{HKST}. Related is a counterpart of the Sobolev space on a domain $\Om\subset X$, the so-called \emph{Newtonian space} $N^{1,p}(X)$ and  the local Newtonian space $N_{loc}^{1,p}(\Om)$,  see~\cite[Chapter 7]{HKST} and \cite{bb} for precise definitions of these spaces and their further properties.

If $X=\R^n$ with the Lebesgue measure and the Euclidean distance, then $N^{1,p}(\Om)$ (resp. $N_{loc}^{1,p}(\Om)$) coincides with the Sobolev space $W^{1,p}(\Om)$ (resp. $W_{loc}^{1,p}(\Om)$) and $g_u=|\nabla u|$.

\subsubsection*{Cheeger energy}\label{sec:Cheeger}

Let $X=(X,d,\mu)$ be a metric measure space and $1\le p <\infty$. The Cheeger \cite{cheeg} energy $Ch_p:L^p(X)\to [0,\infty]$ is defined as the $L^p$-relaxation of the slope.
Shanmugalingam \cite{sh} and Ambrosio--Colombo---Di Marino \cite{acd} showed that for $N^{1,p}$- functions the Cheeger energy coincides with the $p$-integral of the minimal $p$-weak upper gradient $g_u$, i.e. 
\begin{equation}
Ch_p(u)=\frac 1p \int_X g_u^p\ud\mu.
\end{equation}
 Cheeger's definition guarantees that $Ch_p$ is lower semicontinuous with respect to $L^p$-convergence.

\begin{definition}\label{def:ch-harm}
Let $\Omega\subset X$ be a domain in a metric measure space. A continuous function $u\in N^{1,2}_{loc}(\Omega)$ is called harmonic if for any $\varphi\in \LIP_c(\Omega)$ there holds
\begin{align}\label{eq:locmin}
Ch_2(u)\le Ch_2(u+\varphi)
\end{align}
\end{definition}
We assume continuity of Cheeger harmonic functions in the definition. However, in path-connected doubling metric measure spaces supporting Poincar\'e inequality, local H\"older (Lipschitz under stronger assumptions) continuity of Cheeger harmonic functions is established, see~\cite[Theorem 1.1]{ksr}.

\subsubsection*{Korevaar--Schoen energy}
In \cite{Kor-Sch93}, the authors defined a Sobolev space on a metric measure space $X$ by considering what is now called as the \emph{Korevaar--Schoen energy} for $u\in L^2_{loc}(X)$:
\begin{align}\label{eq:ksenergy}
E^2_{KS}(u):=\sup_{\varphi\in C_c(X),0<\varphi\le 1}\limsup_{r\to 0}\int_X\varphi(x)e^2_r(u)\ud\mu,
\end{align}
where the approximate energy densities $e_r^2(u)\ud\mu$ are given by
\begin{align*}
e^2_r(u)(x)=\frac 12\vint_{B_r(x)}\left|\frac{u(y)-u(x)}{r}\right|^2\ud\mu(y), \quad x\in X.
\end{align*}

For each $r>0$, the approximate energy densities arise from a symmetric bi-linear form $e_r:L^2_{loc}(X)\times L^2_{loc}(X)\to L^1_{loc}(X)$ given by
\begin{align*}
e_r(u,v)(x)=\frac 12 \vint_{B_r(x)}\frac{u(y)-u(x)}{r}\frac{v(y)-v(x)}{r}\ud\mu(y),\quad x\in X,
\end{align*}
which satisfies $e_r(u,u)=e^2_r(u).$ Furthermore, we define
\[
 E_r(u,v):=\int_{X} e_r(u,v)\ud\mu.
\]
The Sobolev space $KS^{1,2}(X)$ consisting of functions $u\in L^2(X)$ for which $E_{KS}^2(u)<\infty$, equipped with the norm
\[
\|u\|_{KS^{1,2}}=(\|u\|_{L^2}^2+E^2_{KS}(u))^{1/2}
\]
is known as the \emph{Korevaar--Schoen space}. If $X$ is doubling metric measure space a continuous embedding $M^{1,2}(X)\hookrightarrow KS^{1,2}(X)\hookrightarrow N^{1,2}(X)$ holds true. If $X$ is additionally complete and supports a 2-Poincar\'e inequality the three spaces are equal with equivalent norms, cf. \cite[Theorem 10.5.3]{HKST}.

\subsubsection*{Dirichlet forms}

We follow the notation of \cite[Part I]{dir} and say that a Dirichlet form on $L^2(X)$ is a lower semicontinuous quadratic functional $E:\mathcal F\to [0,\infty]$ defined on a dense domain $\mathcal F\subset L^2(X)$, which is a Hilbert space equipped with the norm
\begin{align*}
\|f\|_{\mathcal F}:=(\|f\|_{L^2}^2+E(f))^{1/2},\quad f\in\mathcal F,
\end{align*}
satisfying
\begin{align*}
E(\min\{f,1\})\le E(f),\quad f\in\mathcal F.
\end{align*}
We denote by $\mathcal E:\mathcal{F}\times \mathcal{F}\to \R$ the bi-linear form associated to $E$ and assume the following strong locality property: \emph{if $f,g\in \mathcal F$  and $g$ is constant on a neighbourhood of $\spt f$, then $\mathcal E(f,g)=0.$}

The form $\mathcal E$ admits a \emph{generator}: a self-adjoint operator $\Delta_{\mathcal E}:D(\Delta_{\mathcal E})\to L^2(X)$ with domain $D(\Delta_{\mathcal E})\subset \mathcal F$ satisfying the following:
\begin{align*}
\int f\Delta_{\mathcal E}g\ud\mu=-\mathcal E(f,g),\quad f\in \mathcal F,\ g\in D(\Delta_{\mathcal E}).
\end{align*}
We call $\Delta_{\mathcal E}$ the laplacian associated to the Dirichlet form. Moreover, if $\mathcal E$ is \emph{regular}, it admits \emph{carr\'e du champ}, i.e. a positive, symmetric and continuous bi-linear operator $\Gamma:\mathcal F\times \mathcal F\to L^1(X)$ with the property that
\begin{align*}
\mathcal E(f,g)&=\int_X\Gamma(f,g)\ud\mu,\quad f,g\in\mathcal F, \textrm{ and }\\
\mathcal E(hf,g)+\mathcal E(f,hg)&=\mathcal E(h,fg)+2\int_Xh\Gamma(f,g)\ud\mu,\quad  f,g,h\in \mathcal F\cap L^\infty(X).
\end{align*}

We may localize a given Dirichlet form $E$ on $L^2(X)$ to a domain $\Omega\subset X$. We say that $f\in \mathcal F_{loc}(\Omega)$ if there is a sequence $(f_k)\subset \mathcal F$ and open sets $\Omega_k\subset \Omega$ such that $\Omega_k\subset \spt(f_k)$, $\Omega=\bigcup_k\Omega_k$, and $f_k=f$ on $\Omega_k$. By the strong locality, we may define the Carr\'e du Champ
\begin{align*}
	\Gamma:\mathcal F_{loc}(\Omega)\times \mathcal F_{loc}(\Omega)\to L^1_{loc}(\Omega),\quad \Gamma(f,g)=\Gamma(f_k,g_k)\textrm{ on }\Omega_k,
\end{align*}
where $\bigcup_k\Omega_k=\Omega$ and $f=f_k$ and $g=g_k$ on $\Omega_k$. We can correspondingly localize $\Delta_{\mathcal E}$, and we denote by $D(\Delta_{\mathcal E};\Omega)$ the domain of $\Delta_{\mathcal E}$ localized to $\Omega$, cf. \cite[Section 7]{dir}. 

The Cheeger and Korevaar--Schoen energies do not define Dirichlet forms in the sense defined above in general. Indeed, while the Cheeger energy is lower semicontinuous, it need not be a quadratic form. The Korevaar--Schoen energy is neither quadratic nor lower semicontinuous in general.

However, in RCD spaces and Carnot groups the Cheeger energy is quadratic and therefore constitutes a Dirichlet form. In particular in these classes the Korevaar--Schoen energy is known to be a constant multiple of the Cheeger energy and thus a posteriori defines a Dirichlet form, see \cite{gt20}. We refer to Section \ref{sec:weak-amv} for a more precise discussion.


\subsubsection*{Spaces with a Korevaar--Schoen Dirichlet form}

The Korevaar--Schoen energy has a distinguished role in this paper due to its close connection with the symmetrized $r$-laplacian, see Section \ref{sec:modified}. For the sake of a unified treatment of weakly amv-harmonic functions in Section \ref{sec:weak-amv}, we introduce the following terminology for spaces where the Korevaar--Schoen energy defines a Dirichlet form.

\begin{definition}\label{def:ksexist}
	Let $X=(X,d,\mu)$ be a locally compact metric measure space. We say that $X$ has a Korevaar--Schoen Dirichlet form (abbreviated KS-Dirichlet form), if $E_{KS}$ is lower semicontinuous with respect to $L^2$-convergence and, for all $u\in KS^{1,2}(X)$, there exists an integrable function $e^2(u)$ satisfying
	\begin{align*}
	\lim_{r\to 0}\int_X\varphi e^2_r(u)\ud\mu=\int_X\varphi e^2(u)\ud\mu
	\end{align*}
	for all $\varphi\in C_c(X)$. The function $e^2(u)$ is called the energy density of $u$.
\end{definition}

By the previous discussion, elaborated in Section \ref{sec:weak-amv}, RCD spaces and Carnot groups equipped with certain homogeneous norms have a KS-Dirichlet form. Sturm \cite{st1,sturm2} introduced the class of $MCP(K,N)$-spaces, which includes RCD-spaces and showed that an energy closely related to $E_{KS}^2$ is a Dirichlet form on $MCP(K,N)$. Note that here we do not consider the existence and lower semicontinuity of the limit in the Korevaar--Schoen energy for maps with metric space targets. This fact is known for RCD and MCP-spaces while for Carnot groups it is -- to our knowledge -- an open question.

If $X$ has a KS-Dirichlet form, then the domain of $E_{KS}^2$ is $KS^{1,2}(X)$ and carr\'e du champ is given by
\begin{align*}
e(u,v):=\frac 14 [e^2(u+v)-e^2(u-v)]=\lim_{r\to 0}e_r(u,v),
\end{align*}
where the convergence is in the weak sense of measures. We denote by $\mathcal E_{KS}$ and $\Delta_{KS}$ the bi-linear form, and its generator, associated to $E_{KS}^2$.

\section{Symmetrized $r$-laplacian and amv-harmonic functions}\label{sec:weak-amv}
In this section we prove Theorems \ref{thm:ncRCD} and \ref{cor:carnot}. We begin by relating the Korevaar--Schoen energy to the symmetrized $r$-laplacians, thus establishing the general principle that local minimizers correspond to functions with vanishing symmetrized $r$-laplacian, cf. Proposition \ref{prop:modified}. 

In the case of Carnot groups the symmetrized $r$-laplacian agrees with the $r$-laplacian. Thus Theorem \ref{cor:carnot} follows directly from Proposition \ref{prop:modified} and a second order Taylor expansion, once we establish that Carnot groups have a Korevaar--Schoen Dirichlet form, see Proposition \ref{prop:carnot}.

Finally we consider non-collapsed RCD-spaces. After reviewing their relevant basic properties we show that having vanishing mm-boundary implies that $\Delta_r$ and $\widetilde \Delta_r$ agree asymptotically (Lemma~\ref{lem:nodev}). Since by \cite{gt20} the Korevaar--Schoen energy is a Dirichlet form, Theorem \ref{thm:ncRCD} follows.


\subsection{Symmetrized $r$-laplacian and the Korevaar--Schoen energy}\label{sec:modified}

Throughout this subsection, $X=(X,d,\mu)$ denotes a generic metric measure space. To define the symmetrized $r$-laplacian, we denote by $k_r:X\times X\to \R$ the \emph{symmetric mean value kernel}: 
\begin{equation}\label{eq:symkernel}
k_r(x,y)=\frac 12\chi_{[0,r]}(d(x,y))\left(\frac{1}{\mu(B_r(x))}+\frac{1}{\mu(B_r(y))}\right),\quad x,y\in X, r>0.
\end{equation}
\begin{definition}\label{def:modlap}
Let $X=(X,d,\mu)$ be a metric measure space, and $r>0$. 
The \emph{symmetrized $r$-laplacian} of a function $u\in L^1_{loc}(X)$ is given by
\begin{align*}
\widetilde\Delta_ru(x)=\int_Xk_r(x,y)\frac{u(y)-u(x)}{r^2}\ud \mu(y),\quad x\in X.
\end{align*}
\end{definition}

Let $u\in L_{loc}^1(X)$ and $r>0$. A direct calculation yields the following identity (cf. formula~\eqref{eq:adjointlaplace}).
\begin{equation}\label{eq:modrel}
\widetilde \Delta_ru=\frac 12\left(\Delta_ru+\Delta_r^*u-u\Delta_r^*1\right),\quad x\in X.
\end{equation}
Moreover, we have the following elementary identities (recall the definitions of $e_r$ and $E_r$ in Section~\ref{sec:ks}).
\begin{lemma}\label{lem:selfadjbasic}
	Let $u,v\in L^2_{loc}(X)$ and $r>0$. Then 
	\begin{align*}
	\Delta_r(uv)=u\Delta_rv+2e_r(u,v)+v\Delta_ru
	\end{align*}
	$\mu$-almost everywhere. In particular, it holds that
	\begin{align*}
	&(a) \qquad \qquad \qquad \qquad \int_Xv\widetilde\Delta_ru\ud\mu=-E_r(u,v) \\
	& and\\
	&(b) \qquad\qquad \qquad \qquad \int_Xv(\Delta_ru-\widetilde\Delta_ru)\ud\mu=\frac 12\int_X v(x)\vint_{B_r(x)}\frac{\delta_r(x,y)}{r}\frac{u(y)-u(x)}{r}\ud\mu(y)\ \ud\mu(x)
	\end{align*}
and
	whenever $u,v\in L^2(X)$. Here
\[
\delta_r(x,y)=1-\frac{\mu(B_r(x))}{\mu(B_r(y))}.
\]
\end{lemma}
\begin{proof}
	For $\mu$-a.e. $x\in X$, we have 
	\begin{align*}
	&u(x)\Delta_rv(x)+2e_r(u,v)(x)+v(x)\Delta_ru(x)\\
	=&\vint_{B_r(x)}\frac{u(x)(v(y)-v(x))+(u(y)-u(x))(v(y)-v(x))+v(x)(u(y)-u(x))}{r^2}\ud\mu(y)\\
	=& \vint_{B_r(x)}\frac{u(y)v(y)-u(x)v(x)}{r^2}\ud\mu(y)=\Delta_r(uv)(x),
	\end{align*}
	proving the first identity in the claim. Together with Remark~\ref{rmk: Green}, the identity \eqref{eq:modrel} implies (a): indeed, we have that
\begin{align*}
2\int_Xv\widetilde\Delta_ru\ud\mu&=\int_Xv(\Delta_ru+\Delta_r^*u-u\Delta_r^*1)\ud\mu=\int_X(v\Delta_ru+v\Delta_r^*u-uv\Delta_r^*1)\ud\mu\\
&=\int_X(v\Delta_ru+u\Delta_rv-\Delta_r(uv))\ud\mu=-2\int_Xe_r(u,v)\ud\mu.
\end{align*}
The identity (b) follows directly from the observation that
\begin{align*}
\Delta_ru(x)-\widetilde\Delta_ru(x)=\frac 12\int_{B_r(x)}\left[\frac{1}{\mu(B_r(x))}-\frac{1}{\mu(B_r(y))}\right]\frac{u(y)-u(x)}{r^2}\ud\mu(y).
\end{align*}
\end{proof}

The characterization of local minimizers of the Korevaar--Schoen energy in terms of the symmetrized $r$-laplacians -- in spaces with a KS-Dirichlet form -- is a straightforward consequence of the preceding observations.


\begin{prop}\label{prop:modified}
	Suppose $X$ is a metric measure space where the  Korevaar--Schoen energy $E_{KS}^2$ exists in the sense of Definition \ref{def:ksexist}. Then the following are equivalent for a function $u\in KS^{1,2}_{loc}(\Omega)$ in a domain $\Omega\subset X$:
	\begin{itemize}
	\item[(1)] $u$ is a local minimizer of $E^2_{KS}$ in $\Omega$;
	\item[(2)] for every $\varphi\in KS^{1,2}(\Omega)$ with compact support in $\Omega$ we have
	\[
	\lim_{r\to 0}\int_X\varphi\widetilde\Delta_ru\ud\mu=0.
	\]
\end{itemize}	
\end{prop}
\begin{proof}
Recall that local minimizers of $E^2_{KS}$ in $\Omega$ are characterized by the property that $\mathcal E_{KS}(\varphi,u)=0$ for every $\varphi\in KS^{1,2}(\Omega)$ with compact support in $\Omega$. Recall also, that the carr\'e du champ for $\mathcal E_{KS}$ is $e(u,v)$. By Lemma \ref{lem:selfadjbasic} and the existence of the Korevaar--Schoen energy, this is equivalent with
\[
\lim_{r\to 0}\int_X\varphi\widetilde\Delta_ru\ud\mu=0.
\]
\end{proof}
\begin{remark}
	More generally, we have the following: if $u\in D(\Delta_{KS};\Omega)$, then
	\[
	\int_X\varphi\Delta_{KS}u\ud\mu=-\mathcal E_{KS}(\varphi,u)=\lim_{r\to 0}\int_X\varphi\widetilde\Delta_ru\ud\mu
	\]
	for each $\varphi\in KS^{1,2}(\Omega)$ with compact support in $\Omega$. 
\end{remark}

\subsection{Carnot groups}\label{sec:Carnot} 
%
%
%
%
%

A Carnot group $\mathbb G$ is a connected and simply connected nilpotent Lie group whose Lie algebra $\mathfrak{g}$ is finite dimensional and admits a stratification of step $s\in \N$
\begin{align*}
\mathfrak{g}=V_1\oplus\cdots\oplus V_s
\end{align*}
with the property that $[V_1,V_j]=V_{j+1}$ for $j=1,\ldots, s-1$ and $[V_1,V_s]=\{0\}$.  $v_i:=\dim V_i$ and set $h_0=0$, $h_i=v_1+\cdots +v_i$ and $m=h_s$. By using the exponential map, every Carnot group $\mathbb{G}$ of step $s$ is isomorphic as a Lie group to $(\R^m, \cdot)$, with the group operation given by the Baker--Campbell--Hausdorff formula.

%
 Moreover, the stratification gives rise to a family of dilations $\delta_t:\mathbb G\to \mathbb G$ for $t>0$, given by
\begin{align*}
\delta_t(x_1,\ldots,x_m)=(t^{\sigma_1}x_1,\ldots,t^{\sigma_s}x_m),
\end{align*}
where the homogeneities $\sigma_i$ are defined by $\sigma_j:=i$ for $h_{i-1}<j\leq h_i$.


A continuous function $\rho:\mathbb{G} \to [0,\infty)$ is said to be a \emph{pseudonorm} on $G$, if $\rho$ satisfies the following conditions:
\begin{enumerate}
\item $\rho(\delta_r (x)) = r\rho(x)$ for every $r > 0$ and $x \in G$,
\item $\rho(x) > 0$ if and only if $x \ne 0$,
\item $\rho$ is symmetric, i.e. $\rho(x^{-1})=\rho(x)$ for every $x \in G$.
\end{enumerate}
A pseudonorm defines a pseudodistance on a group. Examples of such pseudodistances include the Kor\'anyi--Reimann (gauge) distance and the following :
\begin{equation}\label{G-norm}
 |x|_{\mathbb{G}}:=|(x^{(1)},\ldots, x^{(s)})|_{\mathbb{G}}:=\Big(\sum_{j=1}^s \|x^{(j)}\|^{\frac{2
	s!}{j}}\Big)^{\frac{1}{2s!}} \qquad \ud_{\mathbb G}(x,y):=|y^{-1}\cdot x|_{\mathbb{G}},
\end{equation}
where $x^{(i)}=(x_{h_{i-1}+1},\ldots, x_{h_i})$ for $i=1,\ldots, s$, and $\|x^{(j)}\|$ stands for Euclidean norm in $\R^{v_j}$. We observe that with the above notation, the coordinates in $\mathbb{G}$ can be given a form $x=(x^{(1)},x^{(2)},\ldots,x^{(s)})$.

The number $Q:=\sum_{i=1}^s i\dim(V_i)$ is called the \emph{homogeneous dimension} of $\mathbb G$. The Lebesgue measure $\mathcal L^m=:\mu$ is the Haar measure on $\mathbb G$. Moreover, for all $t>0$ it holds $\mu(\delta_t(E))=t^Q\mu(E)$, for a measurable $E\subset \mathbb G$. In particular, $\mu(B(x,r))=cr^Q$ for any left-invariant metric ($c$ depends on the choice of metric) and the Hausdorff-dimension of $\mathbb G$ equals $Q$.

Next we come to Sobolev spaces. Let $\Omega\subset\mathbb G$ be a domain and $X=\{X_1,\ldots, X_m\}$ a Jacobian basis of $\mathfrak{g}$, see \cite[Chapter 1.3]{blu}.
%
We define the horizontal Sobolev space $HW^{1,p}(\Omega)$ as space of measurable functions $u\in L^p(\Omega)$ whose distributional horizontal derivatives $X_ju$ are $L^p$-functions. We denote the horizontal gradient of $u$ by 
\begin{align*}
\nabla_Hu=\sum_{j=1}^{v_1}(X_ju)X_j.
\end{align*}
The norm $\|u\|_{1,p}^p=\|u\|_{L^p(\Omega)}^p+\||\nabla_Hu|\|_{L^p(\Omega)}^p$ makes $HW^{1,p}(\Omega)$ into a Banach space. We also define the horizontal Laplacian $\Delta_{\mathbb G}u$ of a $C^2$-function $u:\Omega\subset \mathbb G\to \R$ by
\begin{equation}\label{eq:subelliptic_laplacian}
	\Delta_{\mathbb G}u:= \sum_{j=1}^{v_1} X_j^2 u.
\end{equation}

The various different notions on Sobolev spaces on $\Omega\subset \mathbb G$ agree, since $\mathbb G$ supports a 1-Poincar\'e inequality, cf.~\cite{HajKosk}. We mention here that, for $u\in HW^{1,2}(\Omega)$, the function $|\nabla_H u|$ is a minimal weak upper gradient of $u$ when $\mathbb G$ is equipped with the Carnot--Carath\'eodory metric.

Below we will also use the approximate Pansu differential $Du_x:\mathbb G\to \R$ of a Sobolev function $u\in HW^{1,2}_{loc}(\Omega)$ at almost every $x\in\Omega$, which can be seen as a natural derivative in Carnot groups and plays an important role in function and mapping theory (e.g. in quasiconformal analysis). The following property relating the Pansu differential and the horizontal gradient of a function can be found in \cite[Remark 3.3]{mo-sc}, see also \cite[Corollary 3.5]{capcow}.
 \begin{align*}
Du_x(z)=\langle\nabla_Hu(x),z^{(1)}\rangle,\quad z=(z^{(1)},\ldots, z^{(s)})\in \mathbb G.
\end{align*}

\begin{prop}\label{prop:carnot}
	Let $\Omega\subset\mathbb G$ be a domain a Carnot group and
	$\rho$ be a pseudonorm on $\mathbb G$ of the form
	\begin{align}\label{eq:carnorm}
	\rho(z)=F(\|z^{(1)}\|, z'),\quad z=(z^{(1)},z')\in \mathbb G, 
	\end{align}

	where $z'=(z^{(2)},\ldots,z^{(s)})$ and $F:\R\times \R^{m-v_1}\to \R$ is $C^1$ outside the origin with $\partial_1F>0$. Then for every $u\in HW^{1,2}(\Omega)$ we have
	\[
	e^2(u)=c|\nabla_Hu|^2
	\]
	almost everywhere, where $c$ is a constant depending only on $\rho$, the data of $\mathbb G$, and $e(u)$ is the KS-energy density of $u$, cf. Definition \ref{def:ksexist}.
\end{prop}
\begin{remark}\label{rmk:carnot}
A particular class of homogeneous norms satisfying \eqref{eq:carnorm} is given as follows. Consider a norm $N$ on $\R^s$ with $\partial_1N>0$, and norms $\|\cdot\|_i$ on $V_i$ for $i=2,\ldots, s$, which need not necessarily be the Euclidean ones. The function
\begin{align}\label{eq:homognorm}
\rho(z):=N(\|z^{(1)}\|,\|z^{(2)}\|_2^{1/2},\ldots,\|z^{(s)}\|_s^{1/s}),\quad z=(z^{(1)},\ldots,z^{(s)})\in \mathbb G
\end{align}
is a homogeneous norm. This class includes the Kor\'anyi--Reimann gauge~\eqref{G-norm}. 
%

\end{remark}
\begin{proof}
By the proof of \cite[Theorem 4.1]{tan}, we have that the KS-energy density exists and satisfies
\begin{align*}
e(u)(x)=\vint_{B_\rho}|Du_x(z)|^2\ud z=\vint_{B_\rho}|\langle\nabla_Hu,z^{(1)}\rangle|^2\ud z
\end{align*}
for a.e $x\in \Omega$. The last equality follows from the remark on the Pansu differential above. Here $B_\rho$ is the unit ball with respect to the pseudonorm $\rho$.

Since every two pseudonorms are equivalent (see~\cite[Proposition 5.1.4]{blu} and \cite[Remark 2]{adw}) we have, in particular, that by~\eqref{G-norm} it holds
\[
\{ z\in \mathbb{G}:  |z|_{\mathbb{G}}\leq 1/C\} \subset \{ z \in \mathbb{G}: F(\|z^{(1)}\|,z')\le 1\} \subset \{ z\in \mathbb{G}:  |z|_{\mathbb{G}}\leq C\}
\]
This observation together with the assumptions on $\rho$ imply that, for each $z'\in M$, where $M\subset \R^{m-v_1}$ is a bounded (and closed) set, the function $F(\cdot,z')$ is invertible and thus there exists a continuous function $G:M\to \R$ such that for each fixed $z'$
\begin{align*}
\{ (z^{(1)}, z')\in \mathbb{G}:\ \rho(z^{(1)},z')\le 1 \}=\{ (z^{(1)}, z')\in \mathbb{G}:\|z^{(1)}\|\le G(z') \}.
\end{align*}

Using this and the change of variables, we compute
\begin{align*}
\int_{B_\rho}|\langle\nabla_Hu,z^{(1)}\rangle|^2\ud z^{(1)}=&\int_{M} \left(\int_{\{ \|z^{(1)}\|\le G(z')\}}|\langle\nabla_Hu(x), z^{(1)}\rangle|^2\ud z^{(1)} \right)\ud z'\\
=&\int_{\{ z'\in M: G(z')\geq 0 \}} G(z')^{2+k}\left(\int_{\{ \|z^{(1)}\|\le 1\}}|\langle\nabla_Hu(x),z^{(1)}\rangle|^2\ud z^{(1)}\right)\ud z'\\
=&c\int_{\{ \|z^{(1)}\|\le 1\}}|\langle\nabla_Hu(x),z^{(1)}\rangle|^2\ud z^{(1)},
\end{align*}
where $c$ is a constant. Indeed, since $G$ is continuous and $M$ is bounded, the integral of $G^{2+k}$ is finite. Then, by \cite[Lemma 4.18]{gt20} we have that
 $$
 \int_{\{ \|z^{(1)}\|\le 1\}}|\langle\nabla_Hu(x),z^{(1)}\rangle|^2\ud z^{(1)}=c(v_1)|\nabla_
Hu(x)|^2.
$$ 
This completes the proof.
\end{proof}

We finish the subsection by providing the proof of Theorem \ref{cor:carnot}.

\begin{proof}[Proof of Theorem \ref{cor:carnot}]
By Proposition \ref{prop:carnot} we have that the energy density $e^2(u)$ exists and satisfies \begin{equation}\label{eq:en_dens_hor_grad}
e^2(u)=c|\nabla_Hu|^2
\end{equation}
for every $u\in HW^{1,2}_{loc}(\Omega)$. Since $|\nabla_Hu|$ is the minimal weak upper gradient of $u$ for the metric measure space $(\Omega,\Ha^Q,d_{cc})$ (cf. \cite[Proposition 11.6 and Theorem 11.7]{HajKosk} for smooth functions from which the claim follows by approximation) it follows that $u\mapsto E_{KS}^2(u)=\int_\Omega e^2(u)\ud\Ha^Q$ is lower semicontinuous and thus the KS-energy is a Dirichlet form in the sense of Definition \ref{def:ksexist}. Moreover by \cite[Part I, Section 1.5 (A5)]{blu} $\Delta_\mathbb G$ is the self-adjoint operator associated to the quadratic form $u\mapsto \int_\Omega|\nabla_Hu|^2\ud z$. This and \eqref{eq:en_dens_hor_grad} imply that $\Delta_{KS}$ is a constant multiple of the sublaplacian $\Delta_\mathbb G$.

Since $\Delta_ru=\widetilde{\Delta}_ru$ for small enough $r$ in any domain compactly contained in $\Omega$ (by the translation invariance of $\Ha^Q$), Proposition \ref{prop:modified} implies that $u\in HW^{1,2}_{loc}(\Omega)$ is weakly amv-harmonic if and only if $\Delta_\mathbb Gu=0$, showing that (ii)$\Leftrightarrow$(iii). Since (i)$\Rightarrow$ (ii) trivially, it remains to prove (iii)$\Rightarrow$(i).

If $\Delta_\mathbb Gu=0$ then $u\in C^\infty(\Omega)$. Now we may argue exactly as in the proof of \cite[Lemma 4.1]{fp} to obtain 
\begin{equation}\label{eq:strong_amv}
u(x)=u_{B_r(x)}+Cr^2\Delta_\mathbb Gu+o(r^2)
\end{equation}
using the second order Taylor expansion \cite[Lemma 2.7]{fp}: the first order terms cancel out since balls $B_r(x)$ with respect to $\rho$ are symmetric, and the computation(s) in the proof of \cite[Lemma 4.1]{fp} remain valid with the sets $S(\varepsilon,x^{(2)},\ldots,x^{(s)})$ (in the notation of \cite{fp}) replaced by the sets $\{ \|z^{(1)}\|\le G(z')\}$ from the proof of Proposition \ref{prop:carnot}. We obtain the constant
\begin{align*}
	C:=\frac{1}{2v_1}\vint_{B_\rho}\|z^{(1)}\|^2\ud z
\end{align*}
in \eqref{eq:strong_amv} where $B_\rho$ stands for the unit ball of $\rho$.

Finally we note that the error term $o(r^2)$ in \eqref{eq:strong_amv} is locally uniform in $x$ (since the same is true for the second order Taylor expansion of a smooth function). This implies that $\Delta_ru(x)\to 0$ as $r\to 0$ locally uniformly in $x\in \Omega$, establishing (iii)$\Rightarrow$(i).

\end{proof}

\subsection{RCD-spaces}\label{sec:rcd}
RCD-spaces arise from the study of synthetic curvature bounds, initially stated in terms of convexity properties of entropy functionals in Wasserstein spaces for Riemannian manifolds, cf. \cite{villani,st1}. A \emph{Riemannian} theory of the resulting \emph{curvature dimension conditions} ($CD(K,N)$-conditions) is obtained by requiring in addition the quadratic property for the Cheeger energy, see Section \ref{sec:ks}. The notion of RCD-spaces was introduced and substantially developed by Ambrosio--Gigli--Savare, cf. \cite{ags1}. In terms of the $\Gamma$-calculus introduced in Section \ref{sec:ks} the $RCD(K,N)$-condition with parameters $K\in \R$ and $N\ge 1$ on a metric measure space $(X,d,\mu)$ can be characterized by the Bochner inequality
\begin{align*}
\frac 12\int |\nabla f|^2\Delta g\ud\mu\ge \int g\left(K|\nabla f|^2+\langle\nabla f,\nabla \Delta f\rangle + \frac 1N (\Delta f)^2\right)\ud\mu
\end{align*}
for any $f\in D(\Delta)$ with $\Delta f\in W^{1,2}(X)$ and any $g\in L^\infty(X)\cap D(\Delta)$.

A cornerstone of RCD-theory is the \emph{Bishop--Gromov} inequality, due to Sturm \cite{sturm2}. To state it, define $s_{K,N},v_{K,N}:[0,\infty)\to \R$ (for $K\in \R$ and $N\ge 1$) by
\begin{align*}
	s_{K,N}(t)=\left\{
	\begin{array}{clr}
		\sqrt{\frac{N-1}{K}}\sin\left(t\sqrt{\frac{K}{N-1}}\right) &\quad K>0, & \\
		t &\quad K=0, & \qquad v_{K,N}(r):=N\omega_N\int_0^rs_{K,N}(t)^{N-1}\ud t,\\
		\sqrt{\frac{N-1}{-K}}\sinh\left(t\sqrt{-\frac{K}{N-1}}\right) &\quad K<0, & 
	\end{array}
	\right.
\end{align*}
where $\omega_N$ the the Lebesgue measure of the unit ball in $\R^N$.
\begin{theorem}[cf. Theorem 2.3, Remark 5.3 in~\cite{sturm2}]\label{thm:bishop-gromov}
	Let $X$ be an $RCD(K,N)$-space, for $K\in\R$ and $N\ge 1$. Then for any $x\in X$, the function
	\[
	r\mapsto \frac{\mu(B_r(x))}{v_{K,N}(r)}
	\]
	is non-increasing on the interval $\left(0,\pi\sqrt{\frac{N-1}{\max\{K,0\}}}\right)$.
\end{theorem}
Note that $v_{K,N}(r)=\omega_Nr^N+O(r^{N+2})$ (cf. \cite{klp17}). However, while $N$ plays the role of dimension upper bound it need not equal Hausdorff dimension of $X$.

Rectifiability properties of RCD-spaces were established in \cite{mon-naber19} and improved in \cite{brue-semola20}, where the \emph{constancy of dimension} of RCD-spaces was proved.
\begin{theorem}[Theorem 3.8 in \cite{brue-semola20}]\label{thm:bruesemola}
Let $(X,d,\mu)$ be an $RCD(K,N)$-space, $K\in \R$ and $N\ge 1$. There exists a natural number $n\le N$ such that for $\mu$-almost every $x\in X$, $\R^n$ (with the Euclidean norm and Lebesgue measure) is the unique tangent cone of $X$ at $x$.
\end{theorem}
The number $n$ above is called the \emph{essential dimension} of $X$, and $X$ is $n$-rectifiable, implying that $n$ equals the Hausdorff dimension. The essential dimension appears in a result of Gigli--Tyulenev \cite{gt20} which states that the Korevaar--Schoen energy is a constant multiple of the Cheeger energy, and thus a Dirichlet form, on RCD-spaces.

\begin{prop}[cf. Theorem 2.18 and Proposition 4.19 in~\cite{gt20}]\label{prop:gigtyu}
	Let  $X$ be an $RCD(K,N)$-space, for $K\in \R$ and $N\ge 1$, and let $n$ be the essential dimension of $X$. Then there is a constant $c=c(K,n)$ so that, if $\Omega\subset X$ is a domain, then for any $u\in KS^{1,2}(\Omega)$ we have that
	\begin{align*}
	\int_\Omega|\nabla u|^2\ud\mu=cE^2_{KS}(u;\Omega),
	\end{align*}
	where $|\nabla u|$ is the minimal weak upper gradient of $u$.
\end{prop}

We remark that by Proposition \ref{prop:gigtyu} local minimizers of $E^2_{KS}$ are exactly the local minimizers of $Ch_2$. It is known, cf. \cite{ksr}, that local minimizers of $Ch_2$ on a domain $\Omega$ are locally Lipschitz continuous in $\Omega$ and are characterized by the equation $\Delta u=0$. 

Next we concentrate on \emph{non-collapsed} RCD$(K,N)$ spaces $X=(X,d,\Ha^N)$. (For an earlier, slightly different, notion of non-collapsed RCD-spaces, see \cite{kitabeppu17}.) As a Corollary of the results above, non-collapsed RCD$(K,N)$-spaces are $N$-rectifiable and locally Ahlfors $N$-regular. Every point $x\in X$ has a tangent cone (which is an RCD$(0,N)$-space) and on a set of full measure the tangent cone is unique and isometrically isomorphic to $\R^N$. The \emph{singular set} $\mathcal S$, defined as the set of points which have a non-Euclidean tangent cone, can be characterized using the densities
\begin{align}\label{eq:bgdens}
	\theta^N_r(x)=\frac{\Ha^N(B_r(x))}{\omega_Nr^N},\quad \theta^N(x):=\lim_{r\to 0}\theta^N_r(x)
\end{align}
as $\mathcal S=\{ x\in X: \theta_N(x)<1 \}$. Note that by Theorem \ref{thm:bishop-gromov} the limit defining $\theta^N$ exists and satisfies $\theta^N\le 1$ \emph{everywhere}. 

The \emph{boundary} of a non-collapsed RCD-space can defined by using the top dimensional stratum of the singular set.
\begin{definition}[See Remark 3.8 in \cite{gig-dephil18} and Definition 4.2 in \cite{klp17}]\label{def:boundary}
	The boundary $\partial X$ of an RCD$(K,N)$-space $X$ is the closure of the set
	\begin{align*}
		\mathcal S^{N-1}\setminus\mathcal S^{N-2}=\{ x\in X: \ \textrm{ the half space }\R^N_+\textrm{ is a tangent cone at }x \}.
	\end{align*}
\end{definition}

It has recently been proven that $\partial X$ is $(N-1)$-rectifiable, cf. \cite{BNS22}. Moreover it was shown in \cite{BMS22} that if a neighbourhood of a point does not meet the boundary then the \emph{metric measure theoretic boundary} near that point vanishes. 


\begin{definition}[Definition 1.5 in \cite{klp17}]\label{def:mmbound}
	Let $X=(X,d,\Ha^N)$ be $N$-rectifiable. We say that $X$ has vanishing mm-boundary (in an open set $U\subset X$), if the signed Radon measures
	\begin{align*}
		\mu_r:=\frac{1-\theta^N_r(\cdot)}{r}\ud\Ha^N,\quad 0<r\leq 1
	\end{align*}
	are uniformly bounded in the total variation norm (in $U$), and converge weakly to zero (in $U$) as $r\to 0$. 
\end{definition}
The following slight reformulation of \cite[Theorem 1.2]{BMS22} can be obtained by a suitable re-scaling of the metric.
\begin{theorem}[Theorem 1.2 in \cite{BMS22}]\label{thm:boundary_vs_mm-boundary}
	Let $X$ be an RCD$(K,N)$-space. There exists $C=C(K,N)$ such that if $x\in X$ and $B_{Cr_0}(x)\cap \partial X=\varnothing$, then
	\[
	\lim_{r\to\infty}|\mu_r|(B_{r_0}(x))=0.
	\]
\end{theorem}



\bigskip\noindent We next employ Theorem \ref{thm:boundary_vs_mm-boundary} to obtain an asymptotic connection between the $r$-laplacian and the symmetrized $r$-laplacian. For the statement recall that
\[
\delta_r(x,y)=1-\frac{\Ha^N(B_r(x))}{\Ha^N(B_r(y))},\quad x,y\in X,\ r>0.
\]

\begin{lemma}\label{lem:nodev}
	Suppose $X$ is a non-collapsed $RCD(K,N)$-space. Then
	\[
	\lim_{r\to 0}\int_X\varphi(x)\vint_{B_r(x)}\left|\frac{\delta_r(x,y)}{r}\right|\ud\Ha^N(y)\ \ud\Ha^N(x)=0.
	\]
	for every  $\varphi\in C_c(X\setminus \partial X)$.
\end{lemma}
\begin{proof}
	Since $\spt\varphi\subset X\setminus \partial X$ we may cover $\spt\varphi$ by a finite number of balls $B_{r_0}(x_i)$ with $B_{2Cr_0}(x_i)\cap \partial X=\varnothing$, $i=1,\ldots M$. Using a partition of unity we may assume that $\spt\varphi\subset B_{r_0}(x_0)$ with  $B_{2Cr_0}(x_0)\cap \partial X=\varnothing$ for some $x_0\in X$. Note that (for small enough $r>0$) we have 

	\begin{align}\label{eq:deltabound}
	|\delta_r(x,y)|=\left|\frac{(\theta_r^N(y)-1)-(\theta^N_r(x)-1)}{\theta^N_r(y)}\right|\le C'(|\theta^N_r(y)-1|+|\theta^N_r(x)-1|),\quad x,y\in B_{r_0}(x_0),
	\end{align}
	where $C'$ is the Ahlfors regularity constant on $B_{2Cr_0}(x_0)$. Thus 
	\begin{align*}
		&\int_X|\varphi|\vint_{B_r(x)}\left|\frac{\delta_r(x,y)}{r}\right|\ud\mu(y)\ \ud\Ha^N(x)\\
		\le & C'\int_X \varphi \left|\frac{\theta^N_r(x)-1}{r}\right|\ud\Ha^N(x)+C'\int_X|\varphi(x)|\vint_{B_r(x)}\left|\frac{\theta^N_r(y)-1}{r}\right|\ud\Ha^N(y)\ \ud\Ha^N(x)\\
		= & C'\int|\varphi|\ud|\mu_r+C'\int (A^*_r|\varphi|) \ud|\mu_r|.
	\end{align*}
For small enough $r>0$ we have that $\spt(A^*_r|\varphi|)\subset B_{r_0}(x_0)$ and, since $\Ha^N$ is (locally) doubling we further have $\|A^*_r|\varphi|\|_\infty\le C''\|\varphi\|_\infty$. It follows that
\begin{align*}
	\lim_{r\to 0} \left|\int_X\varphi(x)\vint_{B_r(x)}\left|\frac{\delta_r(x,y)}{r}\right|\ud\mu(y)\ \ud\Ha^N(x)\right|\le \lim_{r\to 0}C'''\|\varphi\|_\infty |\mu_r|(B_{r_0}(x_0))=0
\end{align*}
\end{proof}

\begin{corol}\label{cor:amv-sa}
	Let $X$ be a non-collapsed $RCD(K,N)$-space with vanishing mm-boundary. Then
	\begin{align*}
	\lim_{r\to 0}\left|\int_X\varphi(\widetilde\Delta_ru-\Delta_ru)\ud\Ha^N\right|=0
	\end{align*}
	whenever $u,\varphi\in \LIP(\Omega)$ and $u$ or $\phi$ have compact support in $\Omega$.
\end{corol}
\begin{proof}
	By a cut-off argument we may assume that both $\varphi$ and $u$ have compact support in $\Omega$. By Lemma \ref{lem:selfadjbasic} we have that
	\begin{align*}
	\left|\int_X\varphi(\widetilde\Delta_ru-\Delta_ru)\ud\Ha^N\right|\le \int_X|\varphi(x)|\LIP(u)\vint_{B_r(x)}\left|\frac{\delta_r(x,y)}{r}\right|\ud\Ha^N(y)\ \ud\Ha^N(x).
	\end{align*}
	The claim now follows from Lemma \ref{lem:nodev}.
\end{proof}

\begin{proof}[Proof of Theorem \ref{thm:ncRCD}]
	Let $u:\Omega\to \R$ be harmonic. In particular, $u$ is locally Lipschitz. By Lemma~\ref{prop:gigtyu} $u$ is a local minimizer of $E_{KS}^2$ in $\Omega$. Theorem~\ref{prop:modified}  implies that
	\[
	\lim_{r\to 0}\int_\Omega\varphi\widetilde\Delta_ru\ud\mu=0
	\]
	for every $u\in KS^{1,2}(\Omega)$ with compact support in $\Omega$. This and Corollary \ref{cor:amv-sa} imply that
	\[
	\lim_{r\to 0}\int_X\varphi\Delta_ru\ud\mu=0
	\]
	for every $\varphi\in \LIP_c(\Omega)$. 
\end{proof}

\begin{remark}\label{rmk:rcd}
It follows from Corollary \ref{cor:amv-sa} and the argument above that, if $u:\Omega\to \R$ is \emph{a locally Lipschitz continuous} weakly amv-harmonic function, then $u$ is Cheeger harmonic.
\end{remark}


\section{Improving regularity: the BPZ-theorem for amv-harmonic functions}\label{sec:bpz}


In this section we discuss Blaschke-Privaloff-Zaremba type theorems (called BPZ theorems for the sake of brevity). In its classical version, see e.g.~\cite[Theorem 2.1.5]{llo}, the BPZ theorem asserts that given an open set in $\R^n$ a continuous pointwise amv-harmonic function solves locally the Laplace equation. Thus, the pointwise nullity of the amv-harmonic operator $\lim_{r\to 0^+}\Delta_r$ improves the regularity of amv-harmonic functions. Below we show that this is also the case of amv-harmonic functions in the homogeneous Carnot groups of step $2$ and on Riemannian manifolds. Moreover, we show a counterpart of the BPZ theorem in Alexandrov spaces and discuss the  relationship between amv-harmonic functions defined in the weak sense and the Laplacian in Alexandrov spaces.
\smallskip

\subsection{Carnot groups.}

Solutions to the subelliptic Laplace equation in homogeneous Carnot groups are $C^\infty$, see e.g. the discussion in~\cite[Section 5.10]{blu} and the references therein. Using the notation from Section~\ref{sec:Carnot} we define $B(p_0,r)$, an open Kor\'anyi--Reimann ball (a gauge ball) centered at point $p_0\in \mathbb G$ with radius $r>0$ as follows
\[
 B=B(p_0,r):=\{p:\in \mathbb G: |p^{-1}p_0|_{\mathbb G}<r\}.
\]
Since, by \cite[Corollary 1]{cap-gar}, the gauge balls are the $X$-NTA domains in step $2$ Carnot groups $\mathbb G$,
they are by the Wiener criterion regular domains and the Dirichlet problem for the continuous boundary data has the classical smooth solutions for the harmonic sub-elliptic equation in $B\subset \mathbb G$, see the~\cite[Theorem 3.9]{dan} and the discussion on pg. 409 in~\cite{cap-gar}. However, it is shown in \cite{hahu}, that for $\Hein$ for $n\geq 2$ balls in the Carnot-Carath\'eodory distance are not regular at the characteristic points (cf.~\cite[Example 14.4]{bb}). Therefore, due to the approach we take in the proof of Theorem~\ref{BPZ-thm} below, we will restrict our discussion to the case of step $2$ Carnot groups and balls with respect to the homogeneous distance $d_{\mathbb G}$ as in~\eqref{G-norm}. We also remark, that whether gauge balls are $X$-NTA in general, step $s$, Carnot groups remains an open problem, see Conjecture 1 on pg. 429 in~\cite{cap-gar}.

For the next theorem recall the distance $d_\mathbb G$ associated to th Koranyi gauge \eqref{G-norm} and the horizontal Laplacian $\Delta_\mathbb Gu$ \eqref{eq:subelliptic_laplacian} from Section~\ref{sec:Carnot}.




\begin{theorem}[The BPZ theorem for amv-harmonic functions in Carnot groups of step $2$]\label{BPZ-thm}
 Let $\Om\subset \mathbb G$ be a domain in the Carnot group $\mathbb G$ of step $2$ equipped with the gauge distance $d_{\mathbb G}$ and let $f:\Om\to \R$ be a continuous function satisfying
 \[
 \lim_{r\to 0^{+}}\Delta_rf(x)=0,\quad x\in\Om.
 \]
 Then $f$ is a sub-elliptic harmonic function, i.e. for all $x\in \Om$ it holds that 
 \[
  \lim_{r\to 0^{+}}\Delta_r f(x)=C\Delta_{\mathbb G} f(x)=0,
  \]
  where $C$ depends on $v_1$ (the dimension of the horizontal layer $V_1$) and the geometry of $\mathbb G$.
\end{theorem}
 
  The constant $C$ is computed in~\cite[Lemma 4.1]{fp}, see also the proof of Theorem \ref{cor:carnot}. We further remark that the same lemma shows the Blaschke-Privaloff-Zaremba theorem for $C^2$-functions in Carnot groups of step $s$.

\begin{proof}
  In the proof we follow the original idea of Privaloff developed for the setting of $\R^n$, see~\cite[Theorem II]{pri}. Let $p_0\in \Om$ and $B=B(p_0, R)$ be a gauge ball centered at $p_0$ with radius $R>0$ such that $B\subset \Om$. The above discussion allows us to infer that the sub-elliptic Dirichlet problem on $B$ with the boundary data $f$ has the (unique) solution, denoted by $u$, such that $u\in C(\overline{B})$ and $u=f|_{\partial B}$. Set $\phi=f-u$. Then $\phi\in C(\overline{B})$ and $\phi|_{\partial B}\equiv 0$. The assertion will be proven if we show that $\phi\equiv 0$ in $B$. We argue by contradiction. Namely, suppose that there exists $q\in B$ such that $\phi(q)\not=0$ and without the loss of generality we assume that $\phi(q)<0$. Let us define the following function on $B$
  \[
   F(p)=\phi(p)+\frac{\phi(q)}{2}\left(\frac{\|(p^{-1}p_0)^{(1)}\|^2-R^2}{R^2}\right),
  \]
  where $\|(p^{-1}p_0)^{(1)}\|$ stands for the Euclidean length in $\R^{v_1}$ of the horizontal part of the point $p^{-1}p_0\in \Om$ (see~\eqref{G-norm}).
  
\emph {Claim: it holds that $F\in C(\overline{B})$, $F|_{\partial B}\geq 0$ and that $F(q)<0$.} Indeed, if $p\in \partial B(p_0, R)$, then $\|(p^{-1}p_0)^{(1)}\|^4+\|(p^{-1}p_0)^{(2)}\|^2=R^4$ for $(p^{-1}p_0)^{(2)}\in \R^{m-v_1}$ and so
 \[
  \|(p^{-1}p_0)^{(1)}\|^2=\sqrt{R^4-|(p^{-1}p_0)^{(2)}\|^2}\leq R^2.
 \]
 Thus $F(p)\geq 0$. Similarly, we find that 
 \[
 F(q)=\phi(q)\left(\frac{\|(p^{-1}p_0)^{(1)}\|^2+R^2}{2R^2}\right)<0.
 \]
 
  Hence, by the claim, there is $q_{1}\in \overline{B}$ such that $F(q_{1})=\min_{\overline{B}} F$ (in fact, $q_1\in B$). Moreover, $\Delta_r F(q_{1})\geq 0$ for all $r\leq d(q_1, \Hei \setminus B)$ and by direct computations we verify that 
  \[
   \Delta_r F(q_{1})=\Delta_r \phi(q_{1})+\frac{\phi(q)}{2}\Delta_r\left(\frac{\|(p^{-1}p_0)^{(1)}\|^2-R^2}{R^2}\right)(q_{1}).
  \]
 Therefore, upon applying the definition of $\Delta_r$, we arrive at the following estimate
 \begin{equation}\label{BPZ-est11}
 0\leq \Delta_r F(q_{1})=\Delta_r f(q_{1})-\Delta_r u(q_{1})+\frac{\phi(q)}{2}\Delta_r\left(\frac{\|(p^{-1}p_0)^{(1)}\|^2-R^2}{R^2}\right)(q_{1}).
 \end{equation}
Then, it holds
 \begin{align*}
  \frac{\phi(q)}{2}\Delta_{\mathbb G} \left(\frac{\|(p^{-1}p_0)^{(1)}\|^2-R^2}{R^2}\right)(q_{1})= \frac{\phi(q)}{2R^2} \Delta_{\mathbb G} \left((p^{-1}p_0)_{1}^2+\cdots+ (p^{-1}p_0)_{v_1}^2\right)(q_1).
 \end{align*}
  
Recall that, by the classical proposition on the Jacobian basis $\{X_1,\ldots, X_m\}$  (see \cite[Corollary 1.3.9]{blu}), the subelliptic Laplacian at any point  $x\in \mathbb G$ can be written as follows:
$$
\Delta_{\mathbb G} = \sum_{j=1}^{v_1}X_j^2=\sum_{j=1}^{v_1}\left(\partial_{j}+\sum_{i=v_1+1}^{m}a_i^{(j)}(x)\partial_{i}\right)^2,
$$
where $a_i^{(j)}$ are the polynomial coefficients. Therefore, in our case the sub-Laplacian reduces to the Laplacian in $\R^{v_1}$. Namely, we get that
\begin{align}
&\Delta_{\mathbb G}  \left((p^{-1}p_0)_{1}^2+\cdots+ (p^{-1}p_0)_{v_1}^2\right)(q_1) \nonumber \\
&\,\,\,\,=\sum_{j=1}^{v_1}\left(\partial_{j}+\sum_{i=v_1+1}^{m}a_i^{(j)}(q_1)\partial_{i}\right)^2 \left((p^{-1}p_0)_{1}^2+\cdots+ (p^{-1}p_0)_{v_1}^2\right) \nonumber \\
&\,\,\,\,=\sum_{j=1}^{v_1}\left(\partial_{j}+\sum_{i=v_1+1}^{m}a_i^{(j)}(q_1)\partial_{i}\right)\Big( \partial_{j}\left((p^{-1}p_0)_{1}^2+\cdots+ (p^{-1}p_0)_{v_1}^2\right)\Big)\nonumber \\
&\,\,\,\,=\sum_{j=1}^{v_1}\left(\partial_{x_j}+\sum_{i=v_1+1}^{m}a_i^{(j)}(q_1)\partial_{x_i}\right)(2(p_0^{-1}p)_{j})
\quad \left(\hbox{\tiny{by $(p^{-1}p_0)_{1}^2=(p_0^{-1}p)_{1}^2$ and \cite[Proposition 3.2.1]{blu}}}\right) \nonumber \\
&\,\,\,\,=\sum_{j=1}^{v_1}2\partial_{x_j}(p_0^{-1}p)_{j}=2v_1. \label{BPZ-lapl-dist}
\end{align} 

Since $f$ is assumed to be pointwise amv-harmonic and $u\in C^2(B)$, the definition of amv-harmonic functions together with \cite[Lemma 4.1]{fp} imply that upon $r\to 0^{+}$ it holds that $\Delta_r f(q_{1})\to 0$ and $\Delta_r u(q_{1})\to 0$. Moreover, by applying \cite[Lemma 4.1]{fp} again and by~\eqref{BPZ-lapl-dist} we obtain that 
\begin{align*}
\lim_{r\to 0^{+}} \frac{\phi(q)}{2}\Delta_r\left(\frac{\|(p^{-1}p_0)^{(1)}\|^2-R^2}{R^2}\right)(q_{1})
= C\frac{\phi(q)}{2}\Delta_{\mathbb G} \left(\frac{\|(p^{-1}p_0)^{(1)}\|^2-R^2}{R^2}\right)(q_{1})=\frac{Cv_1}{R^2}\phi(q)<0,
\end{align*}
since above we assume that $\phi(q)<0$. In a consequence, we get that $0\leq \frac{Cv_1}{R^2}\phi(q)<0$ contradicting our assumption. The proof of the theorem is completed. 
\end{proof}

\subsection{Riemannian manifolds} For the readers convenience, we sketch the proof of the BPZ-theorem in the Riemannian setting. The argument is analogous to that in the proof of Theorem \ref{BPZ-thm}

\begin{prop}[The Blaschke-Privaloff-Zaremba theorem on manifolds]\label{BPZ-thm2}
	Let $(M, g)$ be an $n$-di\-men\-sio\-nal Riemannian manifold with Ricci curvature bounded below by some $K\in \R$. Let further $\Om\subset M$ be a domain in $M$ and let $f:\Om\to \R$ be a continuous function in $\Om$ satisfying
	\[
	\lim_{r\to 0}\Delta_rf(x)=0,\quad x\in \Om.
	\]
	Then $f$ solves the Beltrami-Laplace equation in $\Om$.
\end{prop}

\begin{proof}
	We follow steps of the proof for Theorem~\ref{BPZ-thm} and argue by contradiction. Fix $z_0\in \Om$ and a ball $B=B(z_0, R)\subset \Om$ for a small enough radius $R>0$, which will be determined below. As in the proof of Theorem~\ref{BPZ-thm}, we denote by $u$ the unique solution to the Laplace--Beltrami Dirichlet problem $\Delta_M u=0$ in $B$, such that $u\in C(\overline{B})$ and $u=f|_{\partial B}$. The existence of such $u$ follows, for instance from the general theory for metric measure spaces, since the Riemannian manifold $M$ with the Ricci curvature bounded from below satisfies the doubling condition and the $1$-Poincar\'e inequality (see~\cite[Chapter A.7]{bb} and~\cite{HajKosk} for details).
	
	Let us set $\phi:=f-u$ and notice that $\phi\in C(\overline{B})$ and $\phi|_{\partial B}\equiv 0$. In order to prove the assertion we show that $\phi\equiv 0$ in $B$. We argue by contradiction and suppose that there exists $y\in B$ such that $\phi(y)\not=0$ and without the loss of generality we assume that $\phi(y)<0$.
	
	Similarly as in the proof of Theorem~\ref{BPZ-thm} we define function $F$ as follows
	\[
	F(z)=\phi(z)+\frac{\phi(y)}{2}\left(\frac{\dist_M(z, z_0)^2-R^2}{R^2}\right),
	\]
	where $\dist_M(\cdot, z_0)$ is a distance function to $z_0\in \Om$ defined in metric $g$ on $M$. The direct computations for the Beltrami-Laplace operator, see  the proof of \cite[Theorem 4.1]{li-book}, allow us to conclude that
	\[
	\Delta_M(\dist_M(z, z_0)^2)= 2+2H_{z_0}(z)\dist_M(z, z_0),
	\]
	where $H_{z_0}(z)$ stands for the mean curvature of $\partial B(z_0, r)$ at $z$ for $r=\dist_M(z, z_0)$. According to the discussion in~\cite{paxu}, it holds that $H_{z_0}(z)=\frac{n-1}{r}-\frac13{\rm Ric}_{z_0}(v,v)r+O(r^2)$, asymptotically as $r\to 0$, where $v\in S^{n-1}\subset T_{z_0}M$ (notice that this formula arises as a perturbation of the mean-curvature formula for a sphere in the zero-curvature case). Therefore, by taking $R$ small enough, and hence small enough $r<R$, we obtain that  $\Delta_M(\dist_M(z, z_0)^2)>c>0$ on $B$. 
	
	
	As in the proof of the previous theorem, we find $y_{m}\in \overline{B}$ such that $F(y_{m})=\min_{\overline{B}} F$ (in fact, $y_m\in B$). In order to get the contradiction we need to conclude that $\Delta_r f(y_{m})\to 0$ and $\Delta_r u(y_m) \to 0$ as $r\to 0^+$, cf. \eqref{BPZ-est11} and the discussion following it. The first convergence follows, as $f$ is amv-harmonic, while the second follows from Claim (1), Chapter 6.32 in~\cite{war}. Indeed, by~\cite{war} solutions of the Laplace--Beltrami equation are smooth (assumptions of~ \cite[Chapter 6.32]{war} require $M$ to be compact, but we need this result locally). Similarly to the Euclidean argument, based on expanding $u$ in the Taylor expansion up to the second order, one can show that $\lim_{r\to 0^+}\Delta_r u (y_m)\leq\frac{1}{2(n+2)}\Delta_M u(y_m)=0$, see~\cite[Corollary 3.2.]{zzz}. 
	
	Thus, by analogous reasoning as in the previous theorem we obtain that 
	\begin{align*}
	0\leq\!\lim_{r\to 0^{+}}\!\frac{\phi(y)}{2}\Delta_r\!\left(\frac{\dist_M(z, z_0)^2-R^2}{R^2}\right)\!(y_{m})
	& \leq \frac{\phi(y)}{4(n+2)}\Delta_{M}\! \left(\frac{\dist_M(z, z_0)^2-R^2}{R^2}\right)(y_{m}) \\
	&\leq \frac{c\phi(y)}{4(n+2)R^2}<0.
	\end{align*}
	This contradicts our assumption that $\phi(y)<0$. Hence, the proof of the theorem is completed.
	
\end{proof}

\subsection{Alexandrov spaces and surfaces}

\subsubsection*{The setting of Alexandrov surfaces} The above discussion for Riemannian manifolds can be adapted to the setting of possibly nonsmooth surfaces, namely for closed Alexandrov surfaces with bounded integral curvature, see e.g. \cite{tr1, tr2, re, ab2}. In particular, this together with recent results by Lytchak-Stadler~\cite[Theorem 1.1]{ls} and Petrunin~\cite{Petrunin} imply also a BPZ theorem for non-collapsed $RCD (K,2)$ spaces.

An Alexandrov surface $(S,d)$ is a topological $2$-manifold with a length metric $d$ inducing the manifold topology which has bounded integral curvature in the sense of \cite[Definition 2.1]{tr3}. Alexandrov surfaces admit a notion of curvature which is a Radon measure, denoted by $\omega$. For this and more details we refer the interested reader to~\cite{tr2, tr3}. Here we mention that there exist several equivalent definitions based on polyhedral approximations or on the Gauss-Bonnet theorem. 
Moreover, let us recall that examples of Alexandrov surfaces encompass, for instance, compact polyhedral surfaces (both Euclidean, hyperbolic, spherical or more generally Riemannian), locally $CAT(0)$ surfaces, compact metric surfaces with lower bounded curvature in the sense of Alexandrov (in particular compact Gromov--Hausdorff limits of surfaces with lower bounded Gaussian curvature, see~\cite{bbi, ls}), and also surfaces with conical singularities, see~\cite{tr2, tr3}. 
 
Let $(S,d)$ be a closed Alexandrov surface (i.e. compact surface without boundary) equipped with a metric $d$ and the curvature measure $\omega$.  We say that $(S,d)$ is without cusps, if $\omega(\{x\})< 2\pi$  for all $x\in S$. In particular, both the $BIC$ surfaces with non-positive curvature measure and all the smooth surfaces are automatically without cusps.

Following \cite{tr1, tr2}, if $(S, h)$ is a closed Riemannian surface, then we denote by $V (S, h)$ the space of functions $v:S\to\R$ such that $\mu=\Delta_hv$ is a (signed) measure. By using $v$ one may define a new distance, denoted by $d_{h,v}$, see~\cite[Proposition 7.1]{tr2}. Then, the following result relates Alexandrov's surfaces to smooth geometry, see~\cite[Theorem 7.3]{tr2}. 
\begin{theorem}[Reshetnyak--Huber]\label{BIC-riem}
	Let $(S,d)$ be a closed surface with bounded integral curvature without cusps. Then, there exist a smooth Riemannian metric $h$ on $S$ and a function $v\in V(S,h)$ such that $d=d_{h,v}$. 
\end{theorem}
In fact, the refined version of this theorem allows us to state that the distance $d$ on $S$ is induced by the (singular) Riemannian metric $e^{2v}h$, see~\cite[Theorem 7.1.2]{re} and~\cite[Theorem 2.6]{ab}. Therefore, since the harmonicity is a conformal invariant in dimension two, this relation between $d$ and $h$ enables us to define harmonic functions on open subsets of $(S,d)$ as the functions which are harmonic with respect to metric $h$. More precisely, we define the Laplace-Beltrami operator of $(S,d)$ by 
\begin{equation}\label{Lapl-Belt-AlexS}
\Delta_S:=\Delta_{(S,d)}:=e^{-2v}\Delta_h.
\end{equation}
\begin{prop}\label{prop:Alex-surf}
Let $(S,d)$ be a closed Alexandrov surface with bounded integral curvature without cusps and $f:\Om\to \R$ be a continuous function on a domain $\Om\subset S$, satisfying
\[
\lim_{r\to 0}\Delta_rf(x)=0,\quad x\in \Om.
\]
Then $f$ solves the Beltrami-Laplace equation $\Delta_S f=0$ in $\Om$.
\end{prop}
The proof is similar to the one for Proposition~\ref{BPZ-thm2} and thus we will comment on the key difference only. As in the Riemannian case we argue by contradiction and localize the discussion on a ball $B(z_0, R)\subset \Om$. Then, one considers functions $\phi$ and $F$ similar as in the proof of the above theorem. By the Reschetnyak--Huber theorem we reduce the discussion to the plane with conformally flat metric and by~\eqref{Lapl-Belt-AlexS} get that
$$
\Delta_{S} \dist_S(z,z_0)^2=e^{-2u}\Delta_{\R^2} \dist(z,z_0)^2=2e^{-2u}>0.
$$
From this point, the reasoning follows strictly the proof of Proposition~\ref{BPZ-thm2}.

\subsubsection*{The setting of Alexandrov spaces}

Another setting where a BPZ type theorem can be proven is the setting of Alexandrov spaces. We say that $M$ is an Alexandrov space, if it is locally compact complete length metric space with the lower curvature bound defined via comparison of geodesic triangles in $M$ to geodesic triangles in space form of the appropriate constant sectional curvature, see for instance~\cite{bh, bbi, os}. The Hausdorff dimension of an Alexandrov space turns out to be always an integer and it defines the dimension of the given Alexandrov space. Moreover, a point in an $n$-dimensional Alexandrov space $M$ is called regular, if its tangent cone is isometric to Euclidean $n$-space with standard metric, see e.g. Section 7.1 in~\cite{ab}. A point which is not regular is called singular and it turns out that the Hausdorff dimension of the set of all singular points in $M$ does not exceed $n-1$. Furthermore, the set of all regular points is contained in a $DC_0$ Riemannian manifold, see Definitions 3.1 and 2.1 in~\cite{ab}.

 Let $\Om\subset M$ be a bounded domain in Alexandrov space $M$. For a semiconcave function $u:\Om\to\R$ (see Preliminaries in~\cite{Zhang-Zhu}) a natural Laplacian $\mathcal{L}_u$ defined via the Dirichlet form is a signed Radon measure and its nonsingular part in the Lebesgue decomposition of $\mathcal{L}_u$ satisfies
\begin{equation}\label{Lapl-Perelman}
\Delta u(x)=n\vint_{\Sigma_x} H_xu(\xi,\xi)d\xi,
\end{equation}
for almost all points $x\in \Om$, where $H_x$ denotes the Perelman hessian while $\Sigma_x$ stands for the space of directions at $x$, see~\cite{Petrunin} and \cite[Preliminaries]{Zhang-Zhu}. One of the key features of hessian $H_x$ is that for semiconcave functions in $\Om$ (or more general $DC$ functions, see~\cite{ab}) one can prove the existence of the second order Taylor expansion based on $H_x$, see \cite[Formula (2.6)]{Zhang-Zhu} and~\cite[Proposition 7.5]{ab}.  

The above observations allow us to state the following variant of the BPZ theorem in Alexandrov spaces. 

\begin{prop}[The Blaschke-Privaloff-Zaremba theorem in Alexandrov spaces]\label{BPZ-thm3}
Let $u:\Om\to \R$ be a semiconcave function defined on a domain $\Om$ of Alexandrov space $M$. Then for all $x$ in the set of regular points of $u$ (i.e. for almost all points in $M$) it holds that
\[
 \lim_{r\to 0^+}\Delta_ru(x)=\frac{1}{2(n+2)} \Delta u(x).
\] 
In particular, a semiconcave function $u$ on $\Om$ is harmonic (in the sense of the Laplacian $\Delta$) if and only if $\lim_{r\to 0^+}\Delta_ru(x)=0$ for every regular point $x$ of $u$.
\end{prop}

The proof of this observation is the immediate consequence of the proof of~\cite[Lemma 2.3]{Zhang-Zhu}. We remark that the theorem shows that for a large class of functions the amv-Laplace harmonic functions coincide with those which are harmonic in Alexandrov spaces with respect to Laplacian $\Delta$. However, notice that semiconcave functions are locally Lipschitz. Therefore, in the setting of Alexandrov spaces a characterization of the amv-harmonic functions stated in Proposition~\ref{BPZ-thm3} holds under stronger assumptions than just a continuity as it is in $\R^n$, Carnot groups and on manifolds (cf., respectively, \cite[Theorem 2.1.5]{llo} and Theorem~\ref{BPZ-thm} and Proposition~\ref{BPZ-thm2}). Finally let us comment that, since Alexandrov spaces without boundary are conjectured to have vanishing mm-boundary (see~\cite[Section 1.3]{klp17}) and this is known to be the case for two-dimensional Alexandrov spaces (Theorem 1.8 in \cite{klp17}), the characterization part of Proposition~\ref{BPZ-thm3} for an Alexandrov space $M$ without the boundary is related to our Theorem~\ref{thm:ncRCD}.

\bibliographystyle{plain}

\end{document}